\documentclass[12pt,a4paper]{amsart}
\usepackage{amsmath,amssymb, amsbsy}
\usepackage{color,psfrag}
\usepackage[dvips]{graphicx}
\usepackage{enumerate}
\usepackage{stackrel}
\usepackage[bookmarks=false]{hyperref}

\usepackage{mathtools}
\mathtoolsset{showonlyrefs}

\newcommand{\ieq}{\begin{equation}}
\newcommand{\eeq}{\end{equation}}
\newcommand{\ieqa}{\begin{eqnarray}}
\newcommand{\eeqa}{\end{eqnarray}}
\newcommand{\ieqas}{\begin{eqnarray*}}
\newcommand{\eeqas}{\end{eqnarray*}}

\theoremstyle{plain}
\newtheorem{theorem}{Theorem} [section]
\newtheorem{corollary}[theorem]{Corollary}

\def\neweq#1{\begin{equation}\label{#1}}
\def\endeq{\end{equation}}

\theoremstyle{definition}

\numberwithin{figure}{section}

\newcommand{\average}{{\mathchoice {\kern1ex\vcenter{\hrule
height.4pt width 6pt depth0pt} \kern-11pt} {\kern1ex\vcenter{\hrule height.4pt width 4.3pt depth0pt} \kern-7pt} {} {} }}

\graphicspath{{Images/}}

\begin{document}

\title[Fractions associated to Ford circles extracted by inclined lines]{Properties and approximations of fractions associated to Ford circles extracted by inclined lines}

\author[M. Elizalde]{Mauricio Elizalde}

\address{Mauricio Elizalde\\ Departamento de Matemáticas Fundamentales, Facultad de Ciencias, Universidad Nacional de Educaci\'on a Distancia\\
Madrid, C.P. 28040, Spain\\
melizalde@mat.uned.es}

\author[C. Guevara]{C\'esar Guevara}

\address{C\'esar Guevara \\ Departamento de Matemáticas, Facultad de Ciencias, Universidad Nacional Aut\'onoma de M\'exico \\
Coyoac\'an, C.P. 04510, Mexico City\\
guevaraces@ciencias.unam.mx }

\keywords{Ford circles, Farey sequences, Farey sum, prime omega function, Euler's totient function, M\"obius function, cardinality of fractions, lattice points, jumps.
\\ \indent 2010 {\it MSC: 11N37, 11P21, 11B57, 52C26, 51M15.}}

\begin{abstract}
We study fractions associated to Ford circles which are extracted by means continuous curves. We show that the extracted fractions have similar properties to Farey sequences, like the Farey sum, and we prove that every ordered sequence that satisfies the Farey sum and has two adjacent fractions, can be extracted from Ford circles through continuous curves. This allows us to relate sequences of fractions that satisfy the Farey sum and continuous curves.
We focus on the fractions $F_{1/m}$ extracted from inclined lines with positive slopes of the form $1/m$ and define \textit{jumps} as the cardinality increments of these fractions with respect to $m$. We relate the expression for every jump to the prime omega function in terms of $m$ and find a cardinality formula related to the M\"obius function, which we approximate with three tractable expressions that grow in a log-linear way considering estimations of sums related to Euler's totient function or the graph of the lattice points corresponding to the fractions $p/q$.
\end{abstract}

\maketitle


\section{Introduction}

Ford circles (Figure~\ref{Figure.All}) are defined in $\mathbb{R}^2$ as a family of circles with center in $\left(\dfrac{p}{q},\dfrac{1}{2q^2}\right)$, and radius ${1}/{2q^2}$, where $p$ and $q$ are relatively prime. We associate each circle to the fraction given by its tangent point $p/q$ in the x-axis, and say that two fractions are adjacent when their corresponding circles are tangent.
L. R. Ford introduced these circles in \cite{ford1938fractions} and showed several of their properties, such as that two fractions $\dfrac{p_1}{q_1}$ and $\dfrac{p_2}{q_2}$ are adjacent if and only if $\vert p_2q_1 - p_1q_2\vert = 1$. The author associated these circles to Farey fractions, continued fractions and defined spheres lying in the upper half-space, touching the complex plane, analogue to the definition of Ford circles.

\begin{figure}[h]
\begin{center}
\includegraphics[
trim = 2cm 2cm 2cm 2cm, clip,
width=12cm]{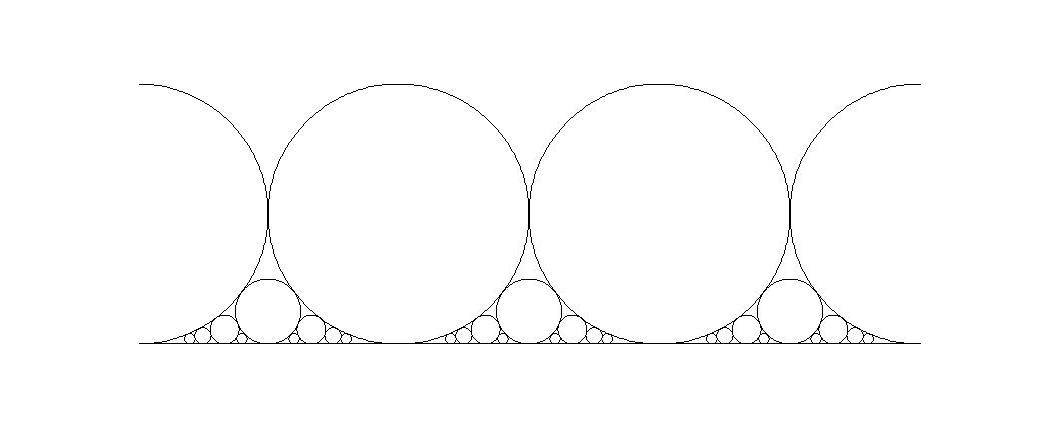}
\end{center}
	\caption{Ford circles.}
\label{Figure.All}
\end{figure}

In his work, Ford described the procedure to \textit{extract} fraction sequences from circles. Consider a continuous curve $L$ crossing the circles defined in the upper half-plane, specifically in $[0,\infty]\times[0,1]$. For every circle $C:\left(\dfrac{p}{q},\dfrac{1}{2q^2}\right)$ that $L$ crosses, we get the fraction $p/q$. If $L$ touches two tangent circles simultaneously, we agree on getting the fraction corresponding to the right circle. We also establish that we must not count a circle twice if the line touches it twice without touching another circle. 
For example, we can extract Farey sequences associated to Ford circles through horizontal lines in the interval $[0,1]$, \textit{e.g.}, if we set $L$: $y=k,\ k\in\mathbb{R^+}$, we get the Farey sequence of order $n$
\begin{equation}
\left\lbrace \ \frac{p}{q} \ : \ p,q\in\mathbb{Z}^+,\ (p,q)=1,\ q\geqslant p,\ q\leqslant n\right\rbrace,
\label{Farey}
\end{equation}
where $n = \Big\lfloor\sqrt{\frac{1}{k}}\Big\rfloor$.

Several authors have found relations between Ford circles and other mathematical objects such as Fibonacci numbers (\cite{kocik2020fibonacci}), the Mandelbrot set (for example \cite{devaney1999mandelbrot} and \cite{tou2017farey}), parabolas (\cite{hutchinson2016parabolas}) and even have
defined the Farey-Ford Polygon (\cite{athreya2015geometry}) and dealt with the geometry of Ford circles (\cite{athreya2014geometric} and \cite{simoson2021ford}). We also think of Ford circles as independent mathematical objects, and we extract fractions from them. 

A natural extension of the fraction extraction with horizontal lines is to consider inclined lines. In this work, we focus on the relation between Ford circles and fractions sequences we can \textit{extract} from inclined lines passing through the origin. At the end of Section \ref{inclined_lines}, we mention the case with lines that do not cross the origin. In that case, horizontal lines (and therefore Farey sequences) are a particular case when the slope tends to zero. We focus on circles in $[0,1] \times [0,1]$  (Figure~\ref{Figure.01}).

\begin{figure}[h]
\begin{center}
\includegraphics[
trim = 0cm 0.2cm 0cm 0cm, clip,
width=7cm]{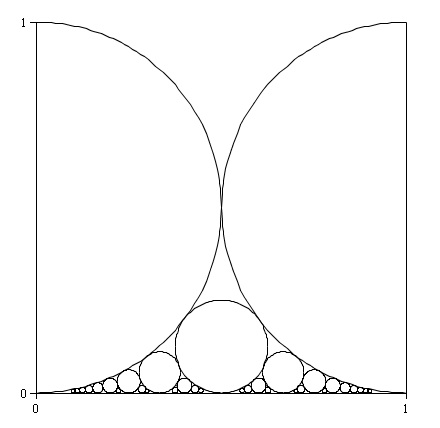}
\end{center}
	\caption{Ford circles in $[0,1] \times [0,1].$}
\label{Figure.01}
\end{figure}


\section{Inclined lines passing through the origin}
\label{inclined_lines}

Let's set the necessary conditions to know which circles are extracted from a line. Given a line $L$, to check if it touches a circle with radius $r$, we see if the minimal distance between the center of the circle and the line is less than or equal to $r$. In the last example for Farey sequences, L is the horizontal line $y=k$, and the condition is that $\left| k-r\right|\leqslant r$, i.e. $\left| k-\dfrac{1}{2q^2}\right|\leqslant\dfrac{1}{2q^2}$ or equivalently, $0\leqslant q^2 \leqslant \dfrac{1}{k}$. 

We now consider a family of lines passing through the origin with small slopes. It does not make sense to consider lines with big slopes since they will touch only one circle. So, we focus on inclined lines of the form $y=\dfrac{1}{m}x,m\in\mathbb{Z}^+$. Note that if $m$ tends to infinity, we have lines tending to the x-axis.\\

\noindent \textit{Claim.} 
Let $F_{1/m}$ be the sequence of fractions extracted from $y=\dfrac{1}{m}x,m\in\mathbb{Z}^+$. 
Then 
\begin{equation}
F_{1/m} = \left\lbrace \ \frac{p}{q} \ :\  p,q\in\mathbb{Z}^+,\ (p,q)=1,\ q\geqslant p,\ pq\leqslant m\right\rbrace.
\label{F1m}
\end{equation}

\begin{proof}
In this case, the circles touched by the line are those such that 
\begin{equation}
\frac{\left|\frac{1}{m}\frac{p}{q} - \frac{1}{2q^2}\right|}{\sqrt{\left(\frac{1}{m}\right)^2+1}}\leqslant r = \frac{1}{2q^2}.
\label{Condition}
\end{equation}
In consequence, 
$$
\dfrac{1}{2}m \left( 1 - \sqrt{\left(\dfrac{1}{m}\right)^2+1}\right) \leqslant pq \leqslant \dfrac{1}{2}m\left( 1 + \sqrt{\left(\dfrac{1}{m}\right)^2+1}\right).
$$ 
The left-hand side of this inequality is negative, and the right-hand side is a non-integer between $m$ and $m+1$, so the condition becomes $pq\leqslant m$.
\end{proof}

For example, for $m=32$ (Figure~\ref{Figure.Line32}), we have that the sequence is
$$F_{1/32} = \left\lbrace \ \frac{p}{q} \ : \  p,q\in\mathbb{Z}^+,\ (p,q)=1,\ q\geqslant p,\ pq\leqslant 32\right\rbrace,$$
%
\begin{figure}[h]
\begin{center}
\includegraphics[
trim = 0cm 1cm 0cm 0cm, clip,
width=12cm]{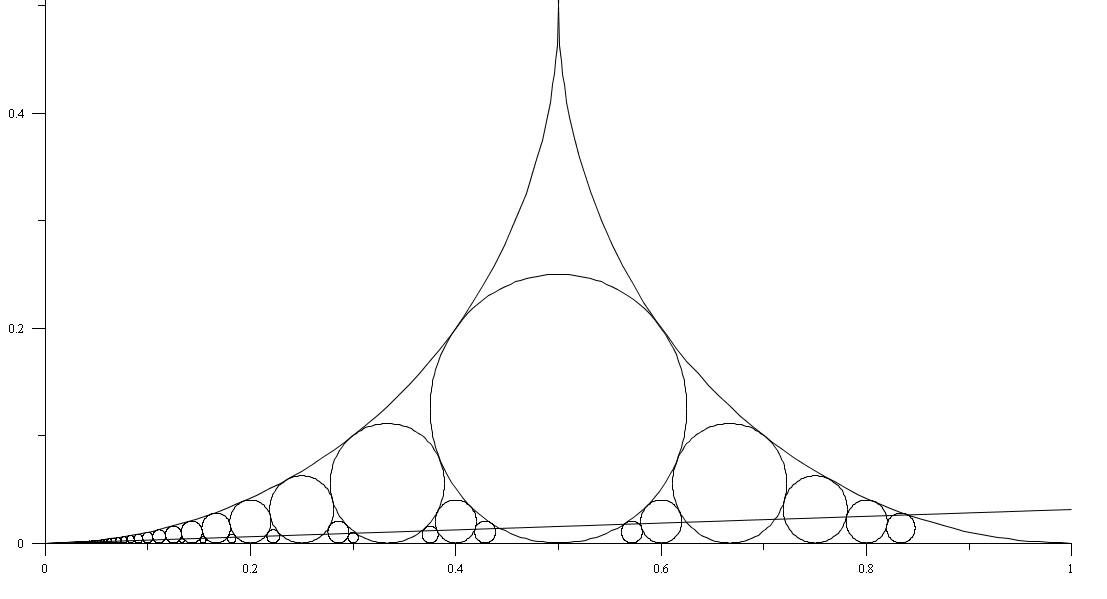}
\end{center}
	\caption{Circles touched by L for $m=32$.}
\label{Figure.Line32}
\end{figure}

and explicitly the sequence is
\begin{equation}
\begin{array}{rl}
\left\lbrace\dfrac{0}{1},\dfrac{1}{32},\dfrac{1}{31},\dfrac{1}{30},... ,\dfrac{1}{9},\dfrac{1}{8},\dfrac{2}{15},\dfrac{1}{7},\dfrac{2}{13},\dfrac{1}{6},\dfrac{2}{11},\dfrac{1}{5},\dfrac{2}{9}, \right. 
\\\\
\left.\dfrac{1}{4},\dfrac{2}{7},\dfrac{3}{10},\dfrac{1}{3},\dfrac{3}{8},\dfrac{2}{5},\dfrac{3}{7},\dfrac{1}{2},\dfrac{4}{7},\dfrac{3}{5},\dfrac{2}{3},\dfrac{3}{4},\dfrac{4}{5},\dfrac{5}{6},\dfrac{1}{1}\right\rbrace.
\end{array}
\label{example}
\end{equation}

To end this section we consider lines that do not pass through the origin, i.e., of the form $y=~\dfrac{1}{m}x+~b,$  $m\in\mathbb{Z}^+,\ b\in(0,1)$. The cardinality of the sequence generated tends to be finite as $m$ grows since the output sequence tends to be a Farey sequence.

Let's see the condition to touch circles from a line $L:$ $y=\dfrac{1}{m}x+b,\ m\in\mathbb{Z}^+,\ b\in(0,1)$. Given $b$, the line crosses a circle if 
\begin{equation}
\frac{\left|\frac{1}{m}\frac{p}{q} - \frac{1}{2q^2} + b\right|}{\sqrt{\left(\frac{1}{m}\right)^2+1}}\leqslant r=\frac{1}{2q^2}.
\label{Condition2}
\end{equation}
Condition \eqref{Condition2} is pretty similar to \eqref{Condition}. In this case, we find $pq+mq^2b\leqslant m$. Therefore, the sequence is 
$$ \left\lbrace \ \frac{p}{q} \ : \  p,q\in\mathbb{Z}^+,\ (p,q)=1,\ q\geqslant p,\ pq+mq^2b\leqslant m \right\rbrace.$$

Observe that as $m$ tends to $\infty$, the condition \eqref{Condition2} is equivalent to $q\leqslant \sqrt{\dfrac{1}{b}}$ and we get the Farey sequence
\begin{equation}
\left\lbrace \ \dfrac{p}{q} \ : \ p,q\in\mathbb{Z}^+,\ (p,q)=1,\ q\geqslant p,\ q\leqslant \sqrt{\frac{1}{b}} \right\rbrace,
\end{equation}
as in \eqref{Farey}.

We have described features of the sequences associated to the extracted circles from inclined lines. Later on, we discuss the cardinality of the sequences generated from lines passing the origin.

\section{Correspondence between Ford circles and the Farey sum}

In \eqref{example} we provided an explicit example of the sequence we extract from inclined lines with $y=x/32$. Observe that if we take three consecutive terms of the sequence $\ \dfrac{a_1}{b_1},\ \dfrac{a_2}{b_2},\ \dfrac{a_3}{b_3}\ $ of $F_{1/32}$, the Farey sum
$ \dfrac{a_2}{b_2} = \dfrac{a_3 + a_1}{b_1 + b_3}, $
holds. In general, this is true for every increasing (or decreasing) sequence of fractions associated to the circles touching a curve, as we show in the following theorem.

\begin{theorem}
Let $L$ be a curve that generates an increasing (or decreasing, without loss of generality) sequence of fractions associated to Ford circles. Then $\ \dfrac{a_2}{b_2} = \dfrac{a_3 + a_1}{b_1 + b_3}\ $ for every consecutive terms $\ \dfrac{a_1}{b_1},\ \dfrac{a_2}{b_2},\ \dfrac{a_3}{b_3}\ $ of the sequence.
\end{theorem}

\begin{proof} 
Assume $L$ generates an increasing sequence $F$ of fractions. Let $\dfrac{a_1}{b_1}$ and $\dfrac{a_2}{b_2}$ be two consecutive fractions of $F$ such that $\dfrac{a_1}{b_1} < \dfrac{a_2}{b_2}$. Since $L$ generates the fractions consecutively, their respective circles are tangent, and the equality $a_2b_1 - a_1b_2 = 1$ holds.

We repeat the argument above with three consecutive fractions of $F$, $\dfrac{a_1}{b_1}, \dfrac{a_2}{b_2}$ and $\dfrac{a_3}{b_3}$ such that $\dfrac{a_1}{b_1} < \dfrac{a_2}{b_2} < \dfrac{a_3}{b_3}$, and find that 
$
a_2b_1 - a_1b_2 = 1 = a_3b_2 - a_2b_3.
$
Therefore, 
$$ \dfrac{a_2}{b_2} = \dfrac{a_3 + a_1}{b_1 + b_3}.
$$
\end{proof} 

Conversely, it is not true that if a sequence $F$ of irreducible fractions satisfies the Farey sum, then $F$ is generated by a curve crossing Ford circles. A counterexample is $\dfrac{9}{14} = \dfrac{4}{7} + \dfrac{5}{7}$: the ordered sequence of these fractions could not be generated by a curve crossing Ford circles, since a curve generates adjacent fractions and therefore tangent circles. The circles associated to these fractions are not tangent to one another. But with a condition on a consecutive pair of fractions we can state the following theorem.

\begin{theorem}
Let $F$ be an increasing (or decreasing) sequence of irreducible fractions such that

\begin{itemize}
\item $\ \dfrac{a_2}{b_2} = \dfrac{a_3 + a_1}{b_1 + b_3}\ $ for all consecutive terms $\ \dfrac{a_1}{b_1},\ \dfrac{a_2}{b_2},\ \dfrac{a_3}{b_3}\ $, and\\

\item At least two Ford circles corresponding to two consecutive fractions of $F$, are tangent.\\
\end{itemize}
Then, $F$ is generated by a curve $L$ crossing the Ford circles, and $L$ is unique when using lines crossing the tangent points and the centers.
\end{theorem}

\begin{proof} 
If a consecutive pair $\dfrac{a_1}{b_1},\dfrac{a_2}{b_2}$ of the fractions satisfy
$a_2b_1 - a_1b_2 = 1$, then this condition plus the Farey sum hypothesis lead to $ a_3b_2 - a_2b_3 = a_2b_1 - a_1b_2 = 1$ for the next fraction $a_3/b_3$ and for the following ones. In this case, we have a set of adjacent fractions, corresponding to tangent Ford circles.

We can construct a curve $L$ that crosses them by joining lines from the tangent points and their centers. Therefore $F$ is generated by a curve crossing the Ford circles. The construction is unique since the tangent points and centers are unique.
\end{proof}

\noindent\textit{Remarks}
\begin{enumerate}
\item The construction mentioned in theorem above is related to the construction of Farey-Ford polygons (see \cite{athreya2015geometry}, where the authors show their construction and related properties), but there could be more curves that touch the same set of circles. 
\item
The fact that at least two circles are tangent, gives us the possibility to find that all circles of those fractions are tangent. Observe that to have this condition, we can ask for the numerators and denominators of their corresponding fractions $\dfrac{a_1}{b_1},\dfrac{a_2}{b_2}$ to be prime relatives, \textit{i.e.,} $(a_1,a_2) =1$ and $ (b_1,b_2)=1$. In the previous counterexample this does not hold since $(7,14)=7$. 
\end{enumerate}

\section{Cardinality of the fractions $F_{1/m}$}

The cardinality of $F_{1/32}$ is 48. It is not immediate to compute cardinalities in general. So far it is clear that if $m$ tends to infinity, the cardinality of $F_{1/m}$ tends to infinity as well. We can see this from the set characterized in \eqref{F1m} and visualize it observing that the circles get ``more dense" when they are closer to the x-axis. We now discuss deeper the cardinality of the fractions extracted from inclined lines.

It is not so easy to find a pattern for the cardinality of the fractions $F_{1/m}$, but if we check the cardinality differences of the consecutive sequences  $F_{1/m}$ and $F_{1/m-1}$, it turns out that they are powers of 2, i.e. 
$$\left| F_{1/m}\right| - \left| F_{1/m-1}\right| = 2^x,$$
for some integer $x \geqslant 0$. We refer to these differences as \textit{jumps} in the cardinality and use the notation 
$S_m := \left| F_{1/m}\right| - \left| F_{1/m-1}\right|$, $m\geqslant$. Note that $S_m$ is the quantity of circles touched by the line with slope $\frac{1}{m}$ but not by the line with slope $\frac{1}{m-1}$. This leads to
\begin{equation}
\begin{array}{llll}
S_m &= \left| F_{1/m} \right| - \left| F_{1/m-1} \right| \\
& = \left| \left\lbrace  \frac{p}{q} : (p,q) = 1,\ q \geqslant p,\ pq \leqslant m \right\rbrace \right| \\
& \ \ \ - \left| \left\lbrace  \frac{p}{q} : (p,q) =1,\ q \geqslant p,\ pq \leqslant m - 1 \right\rbrace \right| 
\\
&= \left| \left\lbrace  \frac{p}{q} : (p,q) = 1,\ q \geqslant p,\ pq = m \right\rbrace \right| \\
&= \left| \left\lbrace (p,q) : \ pq = m, \ (p,q) = 1 \right\rbrace \right|.
\end{array}
\label{S_m_set}
\end{equation}

For counting the pairs $(p,q)$, we consider the last expression in \eqref{S_m_set} and the prime factors of $m$: if $m$ is a prime $p$, then $S_m = \left| \left\lbrace (1,p) \right\rbrace \right| = 1 = 2^0$; if $m$ is a product of two primes $p$ and $q$, $S_m = \left| \left\lbrace (1,pq),(p,q) \right\rbrace \right| = 2 = 2^1$; if $m$ is a product of three primes $p$, $q$ and $r$, $S_m = \left| \left\lbrace (1,pqr),(p,qr),(pq,r),(pr,q) \right\rbrace \right| = 4 = 2^2$; and so on. We observe that the exponent in $S_m$ is equal to the number of the prime factors of $m$ minus $1$. Indeed, it is enough to consider only distinct prime factors. We proof this fact in the following theorem.

\begin{theorem}
Let $F_{1/m}$ be the sequence of fractions extracted from $y=\frac{1}{m}x, \ m \geqslant$, and $S_m$ defined as above. Then
$S_m = 2^{\omega(m)-1}$, where the omega prime function $\omega(m)$ counts the number of distinct prime factors of $m$.
\end{theorem}

\begin{proof}
We consider the squared M\"obius function $\mu^2 (m)$, that equals zero if it exists a prime $p$ such that $p^2 \vert m$, and one otherwise, which means that $m = p_1\cdots p_r$ or $m=1$ . We take the expression $\displaystyle\sum_{d \vert m} \mu^2 (d)$ which sums one when $d$ is a divisor without prime powers greater than one in its factorization. This is equivalent to counting pairs of the form
$$
\begin{array}{l}
(1,p_1p_2p_3\cdots p_k), (p_1,p_2p_3\cdots p_k),(p_2,p_1p_3\cdots p_k), (p_1p_2,p_3\cdots p_k), \ldots ,
\\
(p_1p_2p_3\cdots p_k,1).
\end{array}
$$
Those are the pairs in the last set of \eqref{S_m_set} but here they are counted twice.

Using Bell series (see \cite{apostol1976} p.44), we have that $\displaystyle\sum_{d \vert n} \mu^2 (d) = 2^{\omega(m)}$. Finally we take half of this and we find that
$$S_m = 2^{\omega(m)-1}.$$
\end{proof}

\begin{corollary} 
For $F_{1/m}$  as defined in the above theorem, we have that
$$\left| F_{1/m}\right| = \displaystyle\sum^m_{j=2}2^{\omega(j)-1}+2,\ \  \textup{and} \ \ \left| F_{1/m}\right| = \dfrac{1}{2}\displaystyle\sum^m_{j=2}\sum_{d\vert j}\mu^2(d)+2.$$
\label{Coro}
\end{corollary}

\begin{proof}
From the definition of $S_m$, we write $\left| F_{1/m}\right| = \displaystyle\sum^m_{j=1} S_m +2 \ $ and apply the previous theorem. We have to sum $2$ for the fractions $0/1$ and $1/1$.
\end{proof}

The expressions in Corollary~\ref{Coro} are not easy to handle since they involve the finding of distinct prime factors. In the next section, we approximate the cardinality of $F_{1/m}$ using more closed-form expressions.

\section{Approximations for the cardinality}

In this section, we approximate the cardinality of the fractions $F_{1/m}$ using expression \eqref{F1m}, in three ways. In the first one, for every $p$ in $\{ 1,2,...,s \}$, where $s$ is the maximum possible value of $p$ in $F_{1/m}$, we count the values of $q$ such that $p/q$ is in $F_{1/m}$. In the second one, we bound the lattice area where the possible points $(p,q)$ of $F_{1/m}$ are, and then we approximate the number of points that are relative primes in that region. In the third one, for every $q$ in $\{ 1,2,...,m \}$, we count the values of $p$ in $F_{1/m}$. Observe that the maximum value of $q$ is $m$, that corresponds to the fraction $1/m$.

As we mentioned, we use $s$ as the maximum possible value of $p$. Then, $s$ is the integer that satisfies the inequality 
$$
s(s+1) \leqslant m < (s+1)(s+2).
$$ 
In addition, we have that
$
s = \max \left\lbrace r\in \mathbb{N} : r(r+1) \leqslant m \right\rbrace $ $= \max \left\lbrace r\in \mathbb{N} : r^2+r-m \leqslant 0 \right\rbrace 
= \lfloor - \frac{1}{2} + \frac{1}{2} \sqrt{1+4m} \rfloor.
$
From this, we find 
\begin{equation}
s \ll \sqrt{m} \ll s, \  \ 
O\left(\sqrt{m}\right) = O(s).
\label{m vs s}
\end{equation}
This quantity is important since it gives us the order of magnitude of the errors our approximations have.\\

\subsection{First approximation}

For the first approximation, we fix $p \in \left\lbrace 1,2,...,s \right\rbrace$, and let $C_p$ be the cardinality of the set 
\begin{equation}
\left\lbrace q:(p,q)=1,q\geqslant p,pq\leqslant m \right\rbrace.
\label{F(p)}
\end{equation}  
Then, adding the fractions $0/1$ and $1/1$, we have $\left| F_{1/m} \right| = 2 + \sum^s_{p=1} C_p$.

To compute $C_p$, we take into account the quantity of relative primes to $p$ in the range $\left\lbrace 1,...,x \right\rbrace$, that is $x\dfrac{\varphi(p)}{p}+O(p)$ (see the Appendix in section \ref{Appendix}). We remove from this count the first $\varphi(p)$ numbers such that $q < p$, and we find that
\begin{equation}
\left| F_{1/m} \right| = 2 + \sum^s_{p=1}\left[ m\dfrac{\varphi(p)}{p^2} - \varphi(p)+ O(\varphi(p)) \right].
\label{app1a}
\end{equation} 

Equation \eqref{app1a} is an attempt to have the first approximation. It depends on adding terms involving the Euler's totient function. Then, we use the following identities to build a more direct approximation.
\begin{equation}
\sum^s_{p=1} \dfrac{\varphi(p)}{p^2} = \dfrac{1}{\zeta(2)} (\ln s + C) - A + O\left( \dfrac{\ln s}{s} \right),
\label{eqwithAC}
\end{equation} 
where $C$ is the Euler constant and $A = \displaystyle\sum^\infty_{n=1} \dfrac{\mu(n)\ln(n)}{n^2}$, and
\begin{equation}
\sum^s_{p=1} \varphi(p) = \dfrac{1}{2}\dfrac{s^2}{\zeta(2)} + O(s\ln s).
\label{sumphi}
\end{equation} 

We substitute \eqref{eqwithAC} and \eqref{sumphi} in \eqref{app1a}, and use \eqref{m vs s} to find
$$
\begin{array}{ll}
\left| F_{1/m} \right| &= 2 + m \left[ \dfrac{1}{\zeta(2)} (\ln s+C) - A + O \left( \dfrac{\ln s}{s} \right) \right] 
\\
& \ \ - \left[ \dfrac{1}{2} \dfrac{s^2}{\zeta(2)} + O(s \ln s) \right] + O \left( \displaystyle\sum_{p=1}^s \varphi(p) \right)
\\
&= \dfrac{6m}{\pi^2} \left( \dfrac{s^2}{2m} + \ln s +C \right) - mA + 2  + O (m),
\end{array}
$$
where we use the fact that $\displaystyle\sum^\infty_{k=1}\dfrac{1}{k^2} = \zeta(2)$. Therefore,
\begin{equation}
\left| F_{1/m} \right| = \dfrac{m}{2 \zeta (2)} 
 \left( \dfrac{s^2}{2m} + \ln s +C \right) - mA+2+ O (m).
\label{preapp1}
\end{equation}
We use \eqref{m vs s} to delete $s$ from the expression and we have the first approximation:

\begin{equation}
\left| F_{1/m} \right| = \dfrac{m}{2 \zeta (2)} 
 \left( \ln m + 1 + 2C - 2A \zeta(2) \right) + 2 + O(m),
\label{app1}
\end{equation}
where $C$ and $A$ are given in \eqref{eqwithAC}.

Estimation \eqref{app1} does not depend on a sum and it is more tractable to compute than the expression in Corollary \ref{Coro}.\\

\subsection{Second approximation}

This approximation is inspired on the graph of the lattice points corresponding to the fractions $p/q$, where we represent the variable $p$ in the x-axis and $q$ in the y-axis (see Figure~\ref{Fig.Lattice}). Note that, from the definition, the points $(p,q)\in \mathbb{Z} \times \mathbb{Z}$ representing the fractions $p/q \in F_{1/m}$, are bounded by the lines $p=1$, $q=p$ and the curve $q = m/p$.

\begin{figure}[h]
\includegraphics[scale=.7]{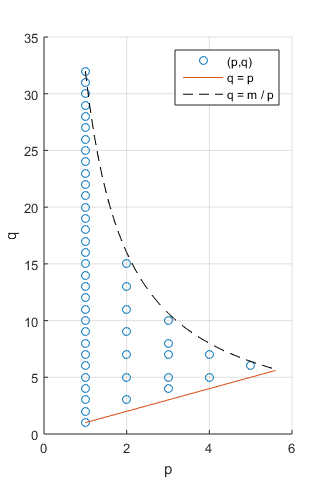}
\centering
\caption{Lattice points $(p,q)$ for the fractions $p/q \in F_{1/32}$. The red line represents the line $q=p$, and the dashed black line represents the curve $q = 32/p$.}
\label{Fig.Lattice}
\end{figure}

We say that two lattice points are \textit{mutually visible} if the segment joining them has no more lattice points. Note that the points $(x,y)$ mutually visible with the origin are those such that $(x,y)=1$, because if there is a point $(x',y')$ in the segment joining $(x,y)$ and $(0,0)$, then $(x,y) = \lambda (x',y'), \ \lambda > 1$. We refer to those points just as \textit{visible points}.

We simplify the approximation problem to counting the visible points in the bounded region described above. For that end, we first consider the squared region $p,q\leqslant r, r \in \mathbb{N}$. Apostol (\cite{apostol1976} p.63) shows that the proportion of the lattice points with respect to the total points tends to $6/\pi^2 \approx .6079$ when $r$ tends to infinity. Therefore, if we take two random integers, the probability they are relative primes is $6/\pi^2$. With this ratio, we can count the total amount of lattice points in the bounded region and multiply it by $6/\pi^2$ to get an approximation.

We consider the number of points below the curve $p\leqslant q/m$ by means of the expression 
$$ \displaystyle\sum^s_{k=1}\dfrac{m}{k} = m \sum^s_{k=1}\dfrac{1}{k}=m H_s,
$$ 
where $H_s$ is the $s$-th harmonic number; and the number of points below the line $q=p$ using the $s$-th triangular number $T_s = \dfrac{s(s+1)}{2}$. Taking into account that $H_s = \ln (s) + C + O\left(\dfrac{1}{s}\right) = \ln (s) + C + O\left(\dfrac{1}{\sqrt{m}}\right)$, the total amount of lattice points in the bounded region is $= m\cdot \ln(s) + mC-T_s + 1 + O\left(\sqrt{m}\right)$, where we added 1 because we subtracted the point $(1,1)$ when we subtracted the ones below $q=p$. Therefore, the amount of visible points in the obtained region is
\begin{equation}
\left| F_{1/m} \right| = 1 + \dfrac{6}{\pi^2}\left( m\cdot \ln(s)+mC-T_s + 1 \right) + O\left(\sqrt{m}\right).
\label{preapp2}
\end{equation}
We use \eqref{m vs s} to delete $s$ from this expression and we have the second approximation:
\begin{equation}
\left| F_{1/m} \right| = 
\dfrac{m}{2 \zeta (2)} (\ln m + 2C - 1) + 1 + \frac{1}{\zeta (2)}
 + O\left(\sqrt{m}\right),
\label{app2}
\end{equation}
where $C$ is the Euler constant.

\subsection{Third approximation}

For our last approximation,  we fix $q$ and consider the values of $p$ that are in the set
\begin{equation}
\left\lbrace p: (p,q) = 1 ,\ p\leqslant q ,\ p \leqslant m/q  \right\rbrace,
\end{equation}
which allows us to estimate the cardinality of $F_{1/m}$ through the sums

$$
\begin{array}{lll}
\left| F_{1/m}\right| &= \displaystyle\sum_{q \leqslant \sqrt{m}}
\left| \left\lbrace p : (p,q) = 1\ ,\ p \leqslant q \right\rbrace \right|
\\
& \ \ + \displaystyle\sum_{\sqrt{m} < q \leqslant m}
\left| \left\lbrace p : (p,q) = 1\ ,\ p \leqslant m/q \right\rbrace \right|
\\
&= \displaystyle\sum_{q \leqslant \sqrt{m}} \varphi (q)\ + 
\sum_{\sqrt{m} < q \leqslant m} \left( m \dfrac{\varphi(q)}{q^2} + O\left(\varphi(q)\right) \right),
\end{array}
$$
where we use the fact that the quantity of relative primes to $q$ in the range $\{1,...,m/q\}$ is (see the Appendix in section \ref{Appendix}) $\dfrac{m}{q}\ \dfrac{\varphi (q)}{q} + O(q)$.

Using equation \eqref{eqwithAC}, we have that
\begin{equation}
\begin{array}{ll}
\displaystyle\sum_{\sqrt{m} < q \leqslant m} \dfrac{\varphi (q)}{q^2} &= \displaystyle\sum_{q \leqslant m} \dfrac{\varphi (q)}{q^2} - \sum_{q \leqslant \sqrt{m}} \dfrac{\varphi (q)}{q^2}\\
& = \dfrac{1}{2\zeta (2)} \ln m + O \left( \dfrac{\ln m}{m} \right).
\end{array}
\label{eq 13}
\end{equation}
Now we use equations \eqref{sumphi} and \eqref{eq 13} to find that
$$
\begin{array}{ll}
\left| F_{1/m} \right| &= \displaystyle\sum_{q \leqslant \sqrt{m}}\varphi (q)\ + \sum_{\sqrt{m}<q \leqslant m} \left( m \dfrac{\varphi (q)}{q^2} + O(q) \right)
\\\\
&= \dfrac{1}{2}\ \dfrac{m}{2\zeta (2)} + O \left( \sqrt{m} \ln \sqrt{m} \right) + m \dfrac{1}{2\zeta (2)} \ln m + O \left( \dfrac{\ln m}{m} \right) + O(m^2).
\end{array}
$$ 

Therefore, our last approximation is given by
\begin{equation}
\left| F_{1/m} \right| = \dfrac{m}{2\zeta (2)} (\ln m + 1) + O(m^2).
\label{app3}
\end{equation}

Despite that this approximation and the first one have a bigger order of magnitude than the second one, in the following section we see that all of them have a common order of magnitude.

\subsection{Comments about the approximations}

We have developed three approximations to the expression given in Corollary \ref{Coro}, which involves the omega prime function, from three different perspectives. We summarize our findings in Table~\ref{Tab.Table1}. We see that these approximations are expressed in closed-form expression, and an interesting fact is that they have the common term
\begin{equation}
\frac{m}{2 \zeta (2)} (\ln m + \alpha ) + \beta,
\end{equation}
where $\alpha$ and  $\beta$ are constants.

\begin{table}[h]
\centering
\caption{Approximations summary.}
\begin{tabular}{|c|l|}
\hline 
Cardinality & $ \left|F_{1/m}\right| = \sum^m_{j=1}2^{\omega(j)-1} +2$ \\ \hline 
Approx. 1 & $ a_1(m) = \frac{m}{2 \zeta (2)} (\ln m + 1 + 2C - 2A \zeta(2)) + 2$ \\
Approx. 2 & $a_2(m) = \frac{m}{2 \zeta (2)} (\ln m + 2C - 1) + 1 + \frac{1}{\zeta (2)}$ \\
Approx. 3 & $a_3(m) = \frac{m}{2 \zeta (2)} (\ln m + 1 )$ \\ \hline 
\end{tabular}
\label{Tab.Table1}
\end{table}


Thus, $\left|F_{1/m}\right|$ grows as $m \ln m$ does, meaning that the cardinality has a log-linear growth. We see empirically from Figure~\ref{Fig.Approx} that the third approximation is the best one, that $\alpha$ must be between $1$ and $1+2C- 2A \zeta(2) \approx 3.2944$ and that $\beta$ must be between $0$ and $2$ to get a good approximation.

\begin{figure}[h]
\begin{center}
\includegraphics[scale=.5]{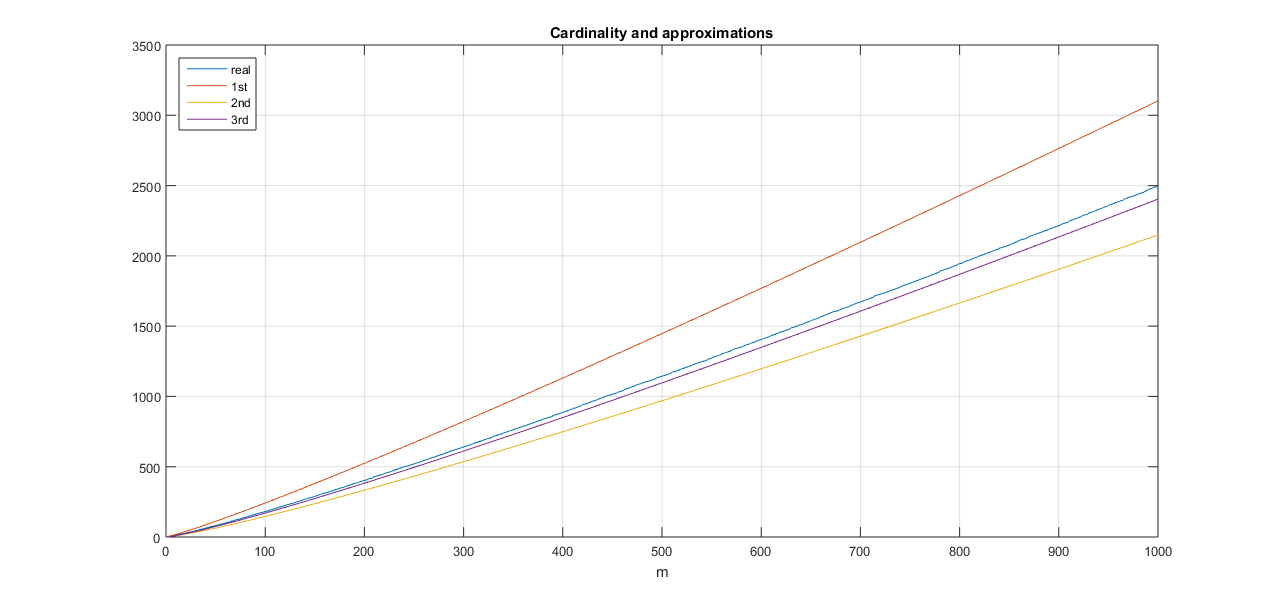}
\centering
\caption{Plot of the real cardinality and the three approximations we developed for $\left| F_{1/m} \right| $.}
\label{Fig.Approx}
\end{center}
\end{figure}

Naming the approximations as in Table \ref{Tab.Table1}, we have that $a_1(m)\ll a_2(m)\ll a_3(m)\ll a_1(m)$, then $\left|\left| F_{1/m}\right| - a_1(m)\right| \ll \left|\left| F_{1/m}\right| - a_2(m)\right| \ll O \left( \sqrt{m} \right)$ and $\left|\left| F_{1/m}\right| - a_3(m)\right| \ll \left|\left| F_{1/m}\right| - a_2(m)\right| \ll O \left( \sqrt{m} \right)$. Therefore, we conclude that the three approximations have the same order of magnitude, which is bounded by $\sqrt{m}$.

\section{Final Remarks}

There are several relations between Ford circles and other mathematical objects. In our case, we relate them to the fractions we can extract by touching them through inclined lines. If the line does not pass through the origin, the result of the extraction is a finite sequence, and when the line is horizontal we achieve the particular case of a Farey sequence. If the line is of the form $y=x/m$, our results show that when we incline ``more'' the slope of the line $1/(m-1)$ to $1/m$, there are new circles touched. The amount of new circles depends on $m$ and its number of distinct prime factors. If the line is horizontal we touch an infinite number of circles. With this information, we have an exact formula for the cardinality of these fractions, but it is in terms of the M\"obius and omega prime functions, and since these functions depend on prime factorization, the formulas do not have closed-form solutions. From this, we develop three ways to estimate the value of the cardinality of the sequences, where we face how to approximate functions involving prime factorization. We find that all out approximations grow in a log-linear way.


\section{Appendix}
\label{Appendix}

This appendix is devoted to approximating the number of relative primes to a fixed positive integer $n$ in the range $\left\lbrace 1,...,x \right\rbrace$, which we denote by $\varphi_x$.

\begin{theorem} Given $n$, $\varphi_x = x\dfrac{\varphi(n)}{n} + O(1)$.  
\end{theorem}

\begin{proof} Observe that if we take a number $x$ between 1 and $n$, the probability that $x$ is relatively prime to $n$ is $\dfrac{\varphi(n)}{n}$. This observation works for $x>n$ taking single residue systems $R$ such that $x \in R$, and is the idea for the proof. 

First, assume that $x$ is a multiple of $n$, then $x=nk$ for some integer $k$. In the range $\{ 1,...,x \}$ there are $k$ consecutive residue systems of $n$, and in each one there are $\varphi (n)$ relative primes to $n$. Therefore, in the whole range there are $\varphi_x  = k\varphi (n) = \dfrac{x}{n}\varphi (n)$. This case is an illustrative example without any error.

Now assume that $x$ is not a multiple of $n$. Let $r$ be the integer such that $(r-1)n<x<rn$. We multiply this inequality by the constant $\dfrac{\varphi(n)}{n}$ to obtain 
\begin{equation} 
(r-1)n\dfrac{\varphi(n)}{n}<x\dfrac{\varphi(n)}{n}<rn\dfrac{\varphi(n)}{n}.
\label{fixapp}
\end{equation}
The bounds of $x \dfrac{\varphi(n)}{n}$ in \eqref{fixapp} represent the number of relative primes to $n$ in the ranges $\{ 1,...,(r-1)n \}$  and $\{ 1,...,rn \}$, respectively. Therefore
\begin{equation}
(r-1)n\dfrac{\varphi(n)}{n} = \varphi_{(r-1)n} \leqslant \varphi_x \leqslant \varphi_{rn} = rn\dfrac{\varphi(n)}{n}.
\label{fix}
\end{equation}

Considering \eqref{fixapp} and \eqref{fix}, we have that
$$\left| \varphi_x - x\dfrac{\varphi(n)}{n} \right| < \left| (r-1)n\dfrac{\varphi(n)}{n} - rn\dfrac{\varphi(n)}{n}\right| = n\dfrac{\varphi(n)}{n} (r-r+1) = \varphi(n),$$ 
and finally 
$\varphi_x = x\dfrac{\varphi(n)}{n} + O(1)$.
\end{proof}


\bibliographystyle{abbrv}
\bibliography{Ford_bib}

\end{document}